\documentclass[12pt,reqno]{article}

\usepackage[usenames]{color}
\usepackage{amssymb}
\usepackage{graphicx}
\usepackage{amscd}

\usepackage[colorlinks=false]{hyperref}

\definecolor{webgreen}{rgb}{0,.5,0}
\definecolor{webbrown}{rgb}{.6,0,0}

\usepackage{color}
\usepackage{fullpage}
\usepackage{float}

\usepackage{graphics,amsmath,amssymb}
\usepackage{amsthm}
\usepackage{amsfonts}
\usepackage{latexsym}
\usepackage{epsf}

\newcommand{\seqnum}[1]{\href{http://oeis.org/#1}{\underline{#1}}}

\begin{document}

\begin{center}
\vskip 0.5cm
\end{center}

\theoremstyle{plain}
\newtheorem{theorem}{Theorem}
\newtheorem{corollary}[theorem]{Corollary}
\newtheorem{lemma}[theorem]{Lemma}
\newtheorem{proposition}[theorem]{Proposition}

\theoremstyle{definition}
\newtheorem{definition}[theorem]{Definition}
\newtheorem{example}[theorem]{Example}
\newtheorem{conjecture}[theorem]{Conjecture}

\theoremstyle{remark}
\newtheorem{remark}[theorem]{Remark}

\begin{center}
\vskip 1cm{\LARGE\bf 
On the geometric mean of the first {\em n} primes 
}
\vskip 0.7cm
\large
Alexei Kourbatov\\
www.JavaScripter.net/math\\
{\tt akourbatov@gmail.com}
\end{center}

\vskip .2 in

\begin{abstract}\noindent
Let $p_n$ be the $n$th prime, and consider the sequence $s_n = (2\cdot3\cdots p_n)^{1/n} = (p_n\#)^{1/n}$,
the geometric mean of the first $n$ primes. We give a short proof that $p_n/s_n \to e$, 
a result conjectured by Vrba (2010) and proved by S\'andor \& Verroken~(2011).
We~show that $p_n/s_n = \exp(1+1/\log p_n + O(1/\log^2 p_n))$ as $n\to\infty$, 
and give explicit lower and upper bounds for the $O(1/\log^2 p_n)$ term.
\end{abstract}

\section{Introduction}
In 2001 A.~Murthy posted OEIS sequence \seqnum{A062049}: 
the integer part of the geometric mean of the first $n$ primes \cite{oeis}.
The sequence is  non-decreasing, unbounded, and begins as follows:
$$
2, 2, 3, 3, 4, 5, 6, 7, 8, 9, 10, 11, 13, 14, 15, 16, 17, 19, 20, 21, 23\ldots
$$
Let $p_n$ be the $n$th prime, and let $s_n$ denote the geometric mean of the first $n$ primes,
$$
s_n = (2\cdot3\cdots p_n)^{1/n} = (p_n\#)^{1/n}, \qquad\mbox{ where }\quad 
p_n\# = 2\cdot3\cdots p_n = \prod_{k=1}^n p_k,
$$
then $\seqnum{A062049}(n) = \lfloor s_n \rfloor$.
(The product $p_n\#$ is called the primorial of $p_n$; see \seqnum{A002110}.)

For many years, sequence \seqnum{A062049} has been lacking an asymptotic formula;
nor did it have any lower or upper bounds for the sequence terms.
In 2010 A.\,Vrba conjectured \cite{rivera} that
$$
p_n/s_n \to e \quad\mbox{ as } n\to\infty.
$$
This was proved in 2011 by S\'andor and Verroken \cite{sandorverroken},
and revisited in 2013 by Hassani \cite{hassani}.

In Section \ref{sec2} we give a new short proof that $p_n/s_n \to e$ and, moreover, we show that
$$
p_n/s_n = \exp(1+1/\log p_n + O(1/\log^2 p_n)).
$$
We give explicit lower and upper bounds for the $O(1/\log^2 p_n)$ term.

\pagebreak
\section{The main result}\label{sec2}

Let $\pi(x)$ denote the prime counting function and $\theta(x)$ denote
Chebyshev's $\theta$ function:
\begin{eqnarray*}
   \pi(x) &=& \sum_{p\le x\atop p \mbox{\tiny\ prime}} 1 ; \\
\theta(x) &=& \sum_{p\le x\atop p \mbox{\tiny\ prime}} \log p
          ~=~ \log \prod_{p\le x\atop p \mbox{\tiny\ prime}}p.
\end{eqnarray*}
Clearly $\pi(p_n)=n$ and $\theta(p_n)=\log(p_n\#)$, so 
$\log s_n = \log(p_n\#)/n = \theta(p_n)/\pi(p_n)$. 

\begin{lemma}\label{lemma1} For $x\ge10^8$ we have
$$
{|\theta(x)-x|\over\pi(x)} ~<~ {1\over\log^2 x}.
$$
\end{lemma}
\begin{proof}
Let $x\ge10^8$. From Dusart \cite{dusart} we have the inequalities 
\begin{eqnarray*}
|\theta(x)-x| &<& {x \over \log^3 x} \ \qquad\mbox{ for } x\ge89967803   \quad\mbox{\rm \cite[Theorem 5.2]{dusart}}, \\
       \pi(x) &>& {x \over \log x-1}    \quad\mbox{ for } x\ge5393 \quad\qquad\mbox{\rm \cite[Theorem 6.9]{dusart}}. 
\end{eqnarray*}
Combining the above inequalities we get
$$
{|\theta(x)-x|\over \pi(x)}
 ~<~ {x \over \log^3 x} \cdot {\log x-1 \over x}
 ~<~ {1 \over \log^2 x}  
$$
for all $x\ge10^8$, as desired.
\end{proof}

\begin{theorem}\label{th2} If $s_n = (p_n\#)^{1/n}$, then $p_n/s_n \to e$ as $n\to\infty$, and for $p_n\ge32059$ we have
\begin{equation}\label{eq1}
\exp\left(1+{1\over\log p_n} + {1.62\over\log^2 p_n}\right) ~<~ p_n/s_n ~<~ 
\exp\left(1+{1\over\log p_n} + {4.83\over\log^2 p_n}\right). 
\end{equation}
\end{theorem}

\begin{proof}
Let $x\ge 10^8$. From Axler \cite[Corollaries 3.5, 3.6]{axler} we have 
$$
\log x - 1 - {1\over\log x} - {3.83\over\log^2 x} ~<~ {x\over\pi(x)} ~<~ 
\log x - 1 - {1\over\log x} - {2.62\over\log^2 x}.
$$
Therefore,
\begin{equation}\label{eq2}
1 + {1\over\log x} + {2.62\over\log^2 x} ~<~ \log x - {x\over\pi(x)} ~<~ 
1 + {1\over\log x} + {3.83\over\log^2 x},
\end{equation}
while
\begin{equation}\label{eq3}
     \log x - {x\over\pi(x)} - {|\theta(x)-x|\over\pi(x)}
 ~<~ \log x - {\theta(x)\over\pi(x)}
 ~<~ \log x - {x\over\pi(x)} + {|\theta(x)-x|\over\pi(x)}.
\end{equation}
Combining (\ref{eq2}) and (\ref{eq3}) with the bound ${|\theta(x)-x|\over\pi(x)}<{1\over\log^2 x}$ 
(Lemma \ref{lemma1}), for $x\ge 10^8$ we get
\begin{equation}\label{eq4}
1 + {1\over\log x} + {1.62\over\log^2 x} ~<~ \log x - {\theta(x)\over\pi(x)} ~<~ 
1 + {1\over\log x} + {4.83\over\log^2 x}.
\end{equation}
But $\log(p_n/s_n)=\log p_n - \theta(p_n)/\pi(p_n)$, so setting in (\ref{eq4})
$x = p_n > 10^8$ we find
\begin{equation}\label{eq5}
1+{1\over\log p_n} + {1.62\over\log^2 p_n} ~<~ \log(p_n/s_n) ~<~ 
1+{1\over\log p_n} + {4.83\over\log^2 p_n}. 
\end{equation}
Exponentiating (\ref{eq5}) we prove the theorem for $p_n>10^8$. Separately, 
we verify by computation that (\ref{eq1}) is true for $32059 \le p_n < 10^8$ as well.
\end{proof}

\noindent
{\bf Remarks.}

\noindent
(i) The convergence $p_n/s_n\to e$ is slow (see Table 1). The better approximation 
\begin{equation}\label{eq6}
p_n/s_n ~\approx~ \exp\left(1+{1\over\log p_n} + {3\over\log^2 p_n}\right)
\end{equation}
has a relative error well below 1\% for $p_n>10^6$, even while $p_n/s_n$ is still far from $e$.

\smallskip\noindent
(ii) One can construct approximations with more\footnote{
The number of terms is meant to be finite, while $p_n$ should be large enough; 
otherwise, such approximations would actually be worse than those with fewer terms.
When $p_n$ is small, even approximation (\ref{eq6}) itself is worse than $p_n/s_n\approx\exp(1+{1\over\log p_n})$ or 
$p_n/s_n\approx e$ \ (see, e.\,g., the first line in Table 1, $p_n=11$).} 
terms:
$$
p_n/s_n ~\approx~ \exp\left(1+{1\over\log p_n} + {3\over\log^2 p_n} + {13\over\log^3 p_n} + \ldots\right),
$$
where the coefficients 1, 3, 13, $\ldots$ are terms of OEIS sequence 
\seqnum{A233824}: a recurrent sequence in Panaitopol's formula for $\pi(x)$
\cite{panaitopol}.
A rigorous proof of such approximations, akin to Theorem \ref{th2}, would depend on sharper bounds for 
${x\over\pi(x)}$ and ${|\theta(x)-x|\over\pi(x)}$,
and these sharper bounds may in turn depend, e.\,g., on the truth of the Riemann Hypothesis.

\begin{center} 
Table 1: \ Values of $n$, $p_n$, $s_n =(p_n\#)^{1/n}$, $p_n/s_n$ 
and approximation (\ref{eq6}) for $p_n\approx10^k$ \\[0.5em]
\begin{tabular}{rrrrcc}
\hline
      $n$ &  $p_n$ & $s_n$ \phantom{\large$11^{1^1}$}   & &
  $p_n/s_n$ & $\exp(1+{1\over\log p_n} + {3\over\log^2 p_n})$
\\
[0.8ex]\hline
\vphantom{\fbox{$1^1$}}
       5 &        11 &         4.706764 & & 2.337062 & 6.950270 \\
      26 &       101 &        29.899069 & & 3.378032 & 3.886576 \\
     169 &      1009 &       298.623420 & & 3.378837 & 3.344393 \\
    1230 &     10007 &      3143.242209 & & 3.183655 & 3.139064 \\
    9593 &    100003 &     32619.709536 & & 3.065723 & 3.032817 \\
   78499 &   1000003 &    334329.282286 & & 2.991072 & 2.968628 \\
  664580 &  10000019 &   3401979.209240 & & 2.939471 & 2.925864 \\
 5761456 & 100000007 &  34435454.560637 & & 2.903984 & 2.895414 \\
50847535 &1000000007 & 347413774.453987 & & 2.878412 & 2.872666 \\
\hline
\end{tabular}
\end{center}

\pagebreak

\noindent(iii)
Bounds (\ref{eq1}) strengthen the double inequality of S\'andor \cite{sandor} 
$$
e ~<~ p_n/s_n ~<~ {p_n \over p_{n-1}} \cdot p_{n+1}^{\pi(n)/n} \qquad\mbox{ for } n\ge10.
$$

\section{Acknowledgments}
I am grateful to all contributors and editors of the websites {\it OEIS.org} and {\it PrimePuzzles.net}, 
particularly to Anton Vrba who conjectured the limit of $p_n/s_n$ \cite{rivera}. 
Thanks also to Christian Axler and Pierre Dusart for proving the $\pi(x)$ and $\theta(x)$ bounds
used in the main theorem.

{\small

\bigskip
\hrule
\bigskip

\noindent 2010 {\it Mathematics Subject Classification}: 11A25, 11N05, 11N37.

\noindent \emph{Keywords: } 
asymptotic formulas, geometric mean, primes, primorial.

\bigskip
\hrule
\bigskip

\noindent (Concerned with sequences
 \seqnum{A002110},
 \seqnum{A062049},
 \seqnum{A233824}.)

\bigskip
\hrule
\bigskip
}
\end{document}